\documentclass[11pt]{article}

\usepackage[utf8x]{inputenc}
\usepackage[margin=3cm]{geometry}
\usepackage{verbatim}
\usepackage{color}
\usepackage[table]{xcolor}
\usepackage{listings}
\usepackage{amssymb,amsfonts,amsthm,amsmath}
\usepackage{url}
\usepackage{enumerate}
\usepackage{todonotes}
\usepackage[all]{xy}
\usepackage{multirow}
\usepackage{changepage}
\usepackage{attachfile}
\usepackage{stmaryrd}
\usepackage{soul}

\usepackage{footnote}
\makesavenoteenv{tabular}
\usepackage{hyperref}

\makeatletter
\newcommand\footnoteref[1]{\protected@xdef\@thefnmark{\ref{#1}}\@footnotemark}
\makeatother

\hypersetup{
	pdfborder={0 0 0}
}

\definecolor{grey}{rgb}{0.95,0.95,0.95}
\definecolor{green}{rgb}{0.2,0.6,0.4}

\newcommand{\N}{\mathbb{N}}

\newcommand{\ext}{\succeq}

\newcommand{\Ccal}{\mathcal{C}}

\newcommand{\cs}{2^\omega}

\newcommand{\uh}{{\upharpoonright}}

\renewcommand{\setminus}{\smallsetminus}

\newcommand{\DNC}{\mathrm{DNC}}

\newcommand{\dnc}{\mathbf{dnc}}
\newcommand{\mlr}{\mathbf{mlr}}
\newcommand{\complex}{\mathbf{complex}}

\newcommand{\req}[1]{\mathcal{R}_{#1}}
\renewcommand{\phi}{\varphi}

\newcommand{\p}[1]{\left( #1 \right)}
\newcommand{\set}[1]{\left\{ #1 \right\}}
\newcommand{\card}[1]{\left| #1 \right|}

\newcommand{\paths}[1]{\llbracket #1 \rrbracket}

\newtheoremstyle{custom}% name of the style to be used
  {10pt}% measure of space to leave above the theorem. E.g.: 3pt
  {10pt}% measure of space to leave below the theorem. E.g.: 3pt
  {\normalfont}% name of font to use in the body of the theorem
  {}% measure of space to indent
  {\bfseries}% name of head font
  {}% punctuation between head and body
  { }% space after theorem head; " " = normal interword space
  {}% Manually specify

\theoremstyle{custom}

\newtheorem{theorem}{Theorem}
\newtheorem{proposition}[theorem]{Proposition}
\newtheorem{lemma}{Lemma}
\newtheorem{definition}{Definition}

\theoremstyle{remark}
\newtheorem*{remark}{Remark}
\newtheorem{claim}{Claim}

\definecolor{lightblue}{rgb}{.60,.60,1}

\definecolor{lightred}{rgb}{1,.60,.60}

\title{Diagonally non-computable functions and fireworks}

\author{
  Laurent Bienvenu \and Ludovic Patey
}

\begin{document}

\maketitle

%\begin{center}
%\includegraphics[width=10cm]{imgs/fireworks-tree.jpg}
%\end{center}
%
%\bigskip

%\bigskip

\begin{abstract}
A set $\Ccal$ of reals is said to be negligible if there is no probabilistic algorithm
which generates a member of~$\Ccal$ with positive probability. Various classes have been proven 
to be negligible, for example the Turing upper-cone of a non-computable real, the class of coherent completions of Peano Arithmetic or the class of reals of minimal Turing degree. 
One class of particular interest in the study of negligibility is the class of diagonally non-computable (DNC) functions, proven by Ku\v cera to be non-negligible in a strong sense: every Martin-L\"of random real computes a DNC function. Ambos-Spies et al.\ showed that the converse does not hold: there are DNC functions which compute no Martin-L\"of random real. In this paper, we show that the set of such DNC functions is in fact non-negligible. More precisely, we prove that for every sufficiently fast-growing computable~$h$, every 2-random real computes an $h$-bounded DNC function which computes no Martin-L\"of random real. Further, we show that the same holds for the set of reals which compute a DNC function but no bounded DNC function. The proofs of these results use a combination of a technique due to Kautz (which, following a metaphor of Shen, we like to call a `fireworks argument') and bushy tree forcing, which is the canonical forcing notion used in the study of DNC functions. 
\end{abstract}

%
%\ludovic{TODO: Motivate the problem by explaining the common intuition about randomness
%and how we break this intuition by giving a natural counterexample. Laurent ?}
%In this document we prove that the measure of oracles which compute a proper d.n.c. function
%-- i.e. a diagonally non-computable function which does not compute any Martin-L\"of random -- is 1. More precisely,
%every 2-random computes a proper DNC function. The proof involves a combination of bushy tree forcing 
%(see~\cite{mushfeq}) and a `fireworks' argument (see~\cite{rumyantsev2013probabilistic}).}

%\clearpage

\section{Background}

\subsection{Negligibility, Levin-V'yugin algebra and DNC functions}

Let~$\Ccal \subseteq \cs$ be a class of infinite binary sequences (a.k.a.\ \emph{reals}) and consider the set
\[
\{X \in \cs: X ~\text{Turing-computes some element of} ~\Ccal\}
\]  
By Kolmogorov's 0-1 law theorem, its (Lebesgue) measure is either~$0$ or~$1$. If it is equal to~$0$, the class $\Ccal$ is said to be~\emph{negligible}, a terminology due to Levin and V'yugin~\cite{LevinV1977}. Equivalently, this means that there is no (infinite) probabilistic algorithm which generates a member of~$\Ccal$ with positive probability. 

Various classes have been proven to be negligible, for example the Turing upper-cone of a non-computable~$A$~\cite{DeLeeuwMSS1956} (or a countable class of such $A$'s), the class of coherent completions of Peano Arithmetic~\cite{JockuschS1972}, the class of reals of minimal Turing degree~\cite{Paris1977}, the class of shift-complex sequences~\cite{Rumyantsev2011,Khan2013}, etc. Likewise, many classes have been showed to be non-negligible (or `typical'). Obviously, classes of positive measure such as the class of Martin-L\"of random reals are all non-negligible. Much more interesting examples were given by Kautz~\cite{Kautz1991}\footnote{Very similar ideas were used by V'yugin in~\cite{Vyugin1976,Vyugin1982}}. He showed that the following are non-negligible: the class of hyperimmune sets, the class of 1-generic reals, and the class of reals of CEA degrees. All three results are variations of the same technique. However, the way Kautz presents his technique is quite abstract and in some sense hides its true spirit, namely that what is being used is a probabilistic algorithm. In the paper~\cite{RumyantsevS2013}, Rumyantsev and Shen give a more explicit and intuitive explanation of the underlying algorithm, using the metaphor of a fireworks shop in which a customer is trying to either buy a box of good fireworks, or expose the vendor by opening a flawed one, and uses a probabilistic algorithm to maximize his chances of success. Following the tradition of colourful terminology in computability theory (Lerman's pinball machine, Nies' decanter argument...), we propose to refer to this method as a \emph{fireworks argument}, the precise template of which will be recalled in the next section. 

The dichotomy between negligibility and non-negligibility has received quite a lot of attention in recent years. We refer the reader to~\cite{BarmpaliasDL2013,BienvenuP2016,Simpson2011} for a panorama of the existing results in this direction.  \\

One class of particular interest in the study of negligibility is the class of diagonally non-computable functions, or `DNC functions' for short. It consists of the total functions~$f: \N \rightarrow \N$ such that $f(n) \not= \phi_n(n)$ for all~$n$, where $(\phi_n)_n$ is a standard enumeration of partial computable functions from $\N$ to~$\N$. By definition, there is no computable DNC function. On the other hand, as showed by Ku\v cera~\cite{Kucera1985}, the class of DNC functions is non-negligible. If we take the point of view of probabilistic algorithms, this is clear: for all~$n$ there is only one value for $f(n)$ to avoid (namely, $\phi_n(n)$ if it is defined), so by picking $f(n)$ at random between~$0$ and some large integer (e.g.\ $2^{n}$), we ensure a positive probability of success. The situation becomes more interesting when one restricts the class $\DNC$ to the subclass
\[
\DNC_h = \{f: \N \rightarrow \N \mid f ~\text{ is a DNC function and }~ (\forall n) f(n) < h(n)\}
\]
where~$h$ is a given function, typically a computable one. The faster~$h$ grows, the easier it is to obtain an element of~$\DNC_h$. And indeed, depending on the growth rate of~$h$ the class~$\DNC_h$ can be negligible or non-negligible (more specifically, for~$h$ computable, $\DNC_h$ is negligible if and only if $\sum_n 1/h(n) = \infty$, this is an unpublished result due to J. Miller, see~\cite{BienvenuP2016} for a proof). 

This notion relativizes to an arbitrary oracle~$X$: a DNC$^X$ function is a function $f$ such that $f(n) \not= \phi^X_n(n)$ for all~$n$. Likewise, we set
\[
\DNC^X_h = \{f: \N \rightarrow \N \mid f ~\text{ is a DNC$^X$ function and }~ (\forall n) f(n) < h(n)\}
\]
Of course, the stronger the oracle~$X$, the harder it is to compute a DNC$^X$ function. \\

In this paper, we study the role of DNC functions in the setting of the Levin-V'yugin algebra, which is the algebra of Turing invariant Borel sets `modulo negligibility'. That is, for two Turing-invariant sets $\mathcal{A}$ and $\mathcal{B}$ we write $\mathcal{A} \sqsubseteq \mathcal{B}$ if $\mathcal{A} \setminus \mathcal{B}$ is negligible. We say that $\mathcal{A}$ and $\mathcal{B}$ are equivalent if $\mathcal{A} \sqsubseteq \mathcal{B}$ and $\mathcal{B} \sqsubseteq \mathcal{A}$. We call an equivalence class for this equivalence relation a \emph{Levin-V'yugin degree}.

Despite having been introduced some time ago and being a very natural notion, little work has been done on the Levin-V'yugin degrees, except for the seminal papers~\cite{LevinV1977,Vyugin1976,Vyugin1982} and some ongoing work by H\"olzl and Porter. It is during discussions with the authors of the latter that a question arose. Ku\v cera's result discussed above shows that reals of Martin-L\"of random degree are also of DNC degree, so $\mlr \sqsubseteq \dnc$, where $\mlr$ is the set of reals Turing-equivalent to a Martin-L\"of random real, and $\dnc$ those equivalent to a DNC function. On the other hand, it is well-known that there are DNC functions which do not Turing-compute any Martin-L\"of random real (see subsection~\ref{subsec:bushy-trees} below). But does this result translate in the setting of the Levin-V'yugin algebra? More precisely, is it then the case that $\mlr \sqsubset \dnc$ (i.e., that $\dnc \setminus \mlr$ is non-negligible)? In this paper, we answer this question in the affirmative. Namely, we prove:

%degrees are a type of `mass problems' degrees: for two Turing-invariant sets $\mathcal{A}, \mathcal{B}$ contained in~$\cs$, we say that $\mathcal{A}$ is V'yugin-reducible to~$\mathcal{B}$, which we denote $\mathcal{A} \leq_V \mathcal{B}$ if $\mathcal{B} \setminus \mathcal{A}$ is negligible. Readers who are familiar with mass problems will see that V'yugin degrees are informally speaking `Muchnik degrees modulo negligibility'. Despite having been introduced some time ago and being a very natural notion, little work has been done on V'yugin degrees, except for the seminal paper~\cite{VYugin1982} and the paper in preparation~\cite{}. It is during discussions with the authors of the latter that a question arose: It is well-known that there are DNC functions which do not Turing-compute any Martin-L\"of random (a result of Ambos-Spies et al~\cite{AmbosSpiesKLS2004}), is it the case that $\mlr <_V \dnc$ where $\mlr$ is the set of sequences Turing-equivalent to a Martin-L\"of random real, and $\dnc$ those equivalent to a DNC function. In other words, is it true that for almost all~$Z \in \cs$, $Z$ computes an element of $\dnc \setminus \mlr$? In this paper, we answer this question in the affirmative. Namely, we prove:  

\begin{theorem}[Main theorem] \label{thm:main-unrelativized}
For every sufficiently fast-growing computable~$h$, every 2-random (i.e., Martin-L\"of random relative to $\emptyset'$) real~$Z$  computes some $f \in \DNC_h$ which does not compute any Martin-L\"of random real.
\end{theorem}

Not only is this result interesting in its own right, but its proof is particularly instructive. It combines fireworks arguments with bushy tree forcing, a forcing notion used in many recent papers to study the properties of DNC functions \cite{AmbosSpiesKLS2004,cai2011elements,dorais2014comparing,greenberg2011diagonally,mushfeq}. To our knowledge, our proof is the most elaborate use of a fireworks argument to date. It illustrates quite convincingly the power of the technique, and is likely to yield further applications in the future.\\

Also interesting is the fact that if we want to study functions which are $\DNC$ relative to some oracle~$X$, we can state a stronger theorem than what we would get from a straightforward relativization of Theorem~\ref{thm:main-unrelativized}.

\begin{theorem}[Main theorem, relativized]\label{thm:main-theorem}
For any real $X$ and a sufficiently fast-growing computable $h$,
every real $Z$ which is both $X$-random and 2-random computes a function $f \in \DNC_h^X$
which itself computes no Martin-L\"of random real.
\end{theorem}

A straightforward relativization of Theorem~\ref{thm:main-unrelativized} would require~$Z$ to be $X'$-random, and would give a function~$f$ which does not compute any $X$-Martin-L\"of random real, but it could still compute a Martin-L\"of random real. Note that taking $X=\emptyset'$ gives us a stronger result than Theorem~\ref{thm:main-unrelativized}: Every 2-random  real~$Z$  computes some $f \in \DNC^{\emptyset'}_h$ which is does not compute any Martin-L\"of random real.
%In particular, for $X = \emptyset$ we get the following statement.
%
%\begin{corollary}
%Every 2-random bounds a non-random DNC degree.
%\end{corollary}

We finally remark \emph{en passant} that `2-random' in the hypothesis of Theorem~\ref{thm:main-theorem} cannot be substituted for `Martin-L\"of random', as shown by the following easy proposition. 
\begin{proposition}\label{prop:easy-ml-is-too-weak}
There is a Martin-L\"of random real which computes no member of $\dnc \setminus \mlr$.
\end{proposition}

\begin{proof}
Take any hyperimmune-free Martin-L\"of random~$Z$. Because of hyperimmune-freeness, everything Turing-computed by~$Z$
is in fact tt-computed by~$Z$. Moreover, every non-computable real which is tt-computed by a Martin-L\"of random real is also of Martin-L\"of random degree~\cite{Demuth1988}. In particular,
every DNC function it computes is of Martin-L\"of random degree.
\end{proof}

\begin{remark}
Readers who are experts in algorithmic randomness may not fully be satisfied with this last proposition, and rightfully so. Indeed, there is a wealth of algorithmic randomness notions between Martin-L\"of randomness and 2-randomness, and it would be interesting to know precisely which levels of randomness it is sufficient for a real to have in order to compute a member of $\dnc \setminus \mlr$. A more precise answer is the following: weak-2-randomness is not sufficiently strong, but Demuth randomness is. For the first part of this assertion, observe that in the proof of Proposition~\ref{prop:easy-ml-is-too-weak}, one can take~$Z$ to be weak-2-random (indeed there are hyperimmune-free weak-2-randoms). The second part is more involved and requires a fine analysis of fireworks arguments which will be done in a forthcoming paper by Christopher Porter and the first author. However, the crude `2-randomness' bound we use in this paper is sufficient for our main goal, which is to prove the non-negligibility of $\dnc \setminus \mlr$. 
\end{remark}

%None of  the techniques we use about bushy trees forcing are new and can be found in the excellent survey~\cite{mushfeq}. However, it is the non-trivial task of combining bushy tree forcing together with a fireworks argument which makes our result interesting. 

\subsection{Notation and terminology}

Unless otherwise specified, a \emph{string} is a finite sequence of integers. We denote the set of strings by $\omega^{<\omega}$, by $\lambda$ the empty string and by 
$\card{\sigma}$ the length of a string~$\sigma$. Unless specified otherwise, a~\emph{sequence} is an infinite list of integers. The set of sequences is denoted by $\omega^{\omega}$. We will sometimes need to consider \emph{binary} sequences (which we also call \emph{reals}), the set of which we denote by $2^\omega$. The $n$-th element of a string or sequence~$Z$ is denoted by $Z(n-1)$ and $Z \uh n$ denotes the finite sequence consisting of the first~$n$ values of~$Z$. A string $\tau$ is a \emph{prefix} of a string $\sigma$ (we also say that $\sigma$ \emph{extends} $\tau$), noted $\tau \preceq \sigma$, if $|\sigma| \geq |\tau|$ and $\sigma \uh |\tau| = \tau$. 

A sequence $Z \in \omega^{\omega}$ is said to be a DNC function (resp.\ $X$-DNC function) if for all~$n$, $Z(n) \not= \phi_n(n)$ (resp.\ $Z(n) \not= \phi_n^X(n)$), where $(\phi_n)$ is a standard enumeration of partial computable functions from $\N$ to $\N$ with oracle.

A \emph{tree} $T \subseteq \omega^{<\omega}$ is a set of strings closed downwards under the prefix relation,
i.e., if $\sigma \in T$ and $\tau \preceq \sigma$ then $\tau \in T$. Members of a tree are often referred to as nodes. A node $\sigma$ is a \emph{child} (or \emph{immediate extension}) of a node $\tau$ in a tree~$T$ if $\tau \preceq \sigma$, $\tau$ and $\sigma$ are nodes of~$T$, and $|\sigma|=|\tau|+1$. A \emph{leaf} of a tree~$T$ is a node with no immediate extension in~$T$. A \emph{path in a tree $T$} is a sequence $f$ such that
every initial segment of $f$ is in $T$. The set of paths of a tree $T$
forms a closed set denoted by $\paths{T}$.

\subsection{Bushy trees}\label{subsec:bushy-trees}
% !TEX root = ../doc.tex

In this section we present the notion of bushy tree and its main properties. Roughly speaking, bushiness is a purely combinatorial property, which states that a tree is `sufficiently fast branching', in a way that guarantees the existence of a DNC path through the tree. The idea of bushy tree was invented by Kumabe, who used it to construct a DNC function of minimal Turing degree (see~\cite{KumabeL2009} for an exposition of this result), which in particular shows that there is a DNC function which computes no Martin-L\"of random real (as no Martin-L\"of random real can have minimal degree). Since then, bushy trees have been successfully applied to the study of DNC functions. We refer the reader to the excellent survey of Khan and Miller~\cite{mushfeq}.

\begin{definition}[Bushy tree]
Fix a function $h$ and a string $\sigma \in \omega^{<\omega}$.
A tree $T$ is  $h$-bushy (resp.\ \emph{exactly $h$-bushy}) above $\sigma$ if every $\tau \in T$
is comparable with $\sigma$ and whenever $\tau \succeq \sigma$ is not a leaf of~$T$,
it has at least (resp.\ exactly) $h(\card{\tau})$ immediate children. 
We call $\sigma$ the \emph{stem} of $T$.
\end{definition}

\begin{definition}[Big set, small set]
Fix a function $h$ and some string $\sigma \in \omega^{<\omega}$.
A set $B \subseteq \omega^{<\omega}$ is \emph{$h$-big above $\sigma$} 
if there exists a finite tree $T$ which is $h$-bushy above $\sigma$ and
such that all leaves of $T$ are in $B$. 
If no such tree exists, $B$ is said to be \emph{$h$-small above $\sigma$}.
\end{definition}

Bushy tree forcing consists generally of building decreasing sequences 
of infinite bushy trees $T_0 \supseteq T_1 \supseteq T_2 \dots$
where $T_i$ is $h$-bushy over $\sigma_i$ for some string $\sigma_i$. 
Each stem $\sigma_i$ is an initial segment of the constructed sequence. 
During the construction we maintain a set $B$ of ``bad'' extensions, 
i.e., of strings to avoid. This set must remain $g$-small above $\sigma_i$ 
at any stage for some function $g$. Bushy tree forcing is especially convenient
for building DNC functions. Let $B_{\DNC}$ be the set of strings which are not initial segments 
of any DNC function:
\[
B_{\DNC} = \set{\sigma \in \omega^{<\omega} : 
  (\exists e < \card{\sigma})\, \phi_e(e) \downarrow = \sigma(e) }
\]

One can easily see that $B_{\DNC}$ is  $2$-small above~$\lambda$. \\

The following three lemmas are at the core of every bushy tree argument. We state them without proof and refer the reader to~\cite{mushfeq} for details. 

\begin{lemma}[Concatenation]\label{lem:concatenation-prop}
Fix a function $h$.
Suppose that $A \subseteq \omega^{<\omega}$ is $h$-big above a string~$\sigma$.
Let $(S_\tau)_{\tau \in A}$ be a family of subsets of  $\omega^{<\omega}$. If $S_\tau \subseteq \omega^{<\omega}$ is $h$-big above $\tau$
for every $\tau \in A$, then $\bigcup_{\tau \in A} S_\tau$ is $h$-big above $\sigma$.
\end{lemma}

The concatenation property is often used in the following contrapositive form:
If we are given a finite tree $T$ $h$-bushy above some string $\sigma$
and a ``bad'' set $B$ of extensions to avoid which is $h$-small above $\sigma$, then there exists
a leaf $\tau$ of $T$ such that $B$ is still $h$-small above $\tau$.
In particular, if a set $A$ is $h$-big above $\sigma$, there exists an extension $\tau$
of $\sigma$ which is in $A$ and such that $B$ is still $h$-small above $\tau$.

\begin{lemma}[Smallness additivity] \label{lem:smallness-add}
Suppose that $B_1, B_2, \ldots, B_n$ are subsets of $\omega^{<\omega}$, $g_1$, $g_2$, ..., $g_n$ are functions, and $\sigma \in \omega^{<\omega}$.
If $B_i$ is $g_i$-small above~$\sigma$ for all~$i$, then $\bigcup_i B_i$ is $(\sum_i g_i)$-small above $\sigma$. 
\end{lemma}

%The big subset property is the key property of bushy tree forcing.
%It enables us to force one particular event when we know we can force a small disjunction of events.
%For example imagine we are given a string $\sigma$ and a bad set $B$ which is $k$-small above $\sigma$.
%We consider the set of extensions which make the 2-valued functional $\Gamma$
%terminate on some input $n$:
%$$
%S = \set{ \tau \succeq \sigma : \Gamma^\tau(n) \downarrow}
%$$
%Then either $S \cup B$ is $2k$-small above $\sigma$ and taking $S \cup B$ as our new bad set,
%we will be able to force $\Gamma$ to diverge on $n$. Or $S \cup B$ is $2k$-big above $\sigma$.
%Consider the set $S_i = \set{ \tau \succeq \sigma : \Gamma^\tau(n) \downarrow = i}$.
%By the big subset property, either $S_0 \cup B$ or $S_1 \cup B$ is $k$-big above $\sigma$.
%Say $S_0 \cup B$. Then by the concatenation property, we are able to find an extension~$\tau$ of $\sigma$ which is in $S_0 \cup B$ but so that $B$ is still $k$-small above $\sigma$. So either we are able to force $\Gamma$ to diverge on input $n$, or we force it to converge in a particular value.

\begin{lemma}[Small set closure]\label{lem:small-set-closure}
%Suppose that $B \subseteq \omega^{<\omega}$ is $g$-small above some string $\sigma \in \omega^{<\omega}$. 
We say that $B  \subseteq \omega^{<\omega}$ is \emph{$g$-closed} if whenever $B$ is $g$-big above a string $\rho$ then $\rho \in B$. Accordingly, the \emph{$g$-closure} of any set~$B \subseteq \omega^{<\omega}$ is the set $C = \set{\tau \in \omega^{<\omega} : B \mbox{ is $g$-big above } \tau}$. If $B$ is $g$-small above a string~$\sigma$, then its closure is also $g$-small above $\sigma$.
\end{lemma}

As explained in~\cite{mushfeq}, Lemma~\ref{lem:small-set-closure} is very useful in our constructions. We are given a ``bad'' set $B$ of nodes, which is $g$-small above $\sigma$ where $\sigma$ is a partial approximation of the object we are constructing. We want to extend $\sigma$ still avoiding $B$ and in particular preserving $g$-smallness
of $B$. Lemma~\ref{lem:small-set-closure} enables us to consider w.l.o.g. that if $\rho$ is an extension of $\sigma$ which
does not preserves $g$-smallness of $B$, then $\rho$ is already on~$B$.\\

The next lemma is very simple and yet central in this paper. It expresses the fact that if a set~$B$ is sufficiently small in an~$h$-bushy tree~$T$, then there is only a small probability that a random path of the tree meets~$B$ (has a member of~$B$ as prefix). By ``random path", we mean the probability distribution over paths induced by a downward random walk where one starts at the root and at each step goes from a node to one of its children, all children being given the same probability of being picked. 
%
%Bushy tree forcing did not used to be applied in probabilistic constructions.
%The above mentioned lemmas mainly care about the existence of extensions avoiding some bad set of events.
%There exists however a bridge between smallness and probability that a randomly chosen extension meets
%a small bad set. Intuitively, a set $B$ being $h$-small above $\sigma$ for a slow-growing function $h$ means that there
%is a small probability to meet $B$ when constructing a sequence extending $\sigma$.
%This intuition is formalized through Lemma~\ref{lem:small-vs-measure}.

%PREVIOUS PROOF ONLY WORKED FOR EXACTLY h-BUSHY
\begin{lemma}\label{lem:small-vs-measure}
Fix two functions positive functions $g$ and $h$ with $g \leq h$. If $T$ is an infinite tree $h$-bushy above $\lambda$
and $B \subseteq T$ is a $g$-small above $\lambda$, then the probability that
a random path of~$T$ avoids $B$ is greater than
\[
\prod_{i \in \omega} \left( 1 - \frac{g(i)}{h(i)} \right)
\]
\end{lemma}

Note that this quantity is positive if and only if $\sum g(i)/h(i) < \infty$, due to the identity $\prod_{i \in \omega}\p{1 - e_i} = \exp\p{- \sum_{i \in \omega} \ln(1-e_i)}$ and the asymptotic estimate $-\ln(1-x) =x+o(x)$. 
\begin{proof}
Without loss of generality, we can assume that $B$ is $g$-closed (otherwise, take its closure). We prove by induction over $n$ that the probability of having avoided~$B$ by the time we reach depth~$n$ is at least $\prod_{i=0}^{n-1} \big(1-g(i)/h(i)\big)$ (a quantity equal to~$1$ for $n=0$, by convention). The lemma immediately follows from this fact. 

The base case $n = 0$ is trivial as~$\lambda$ is the only such node and $B$ is $g$-small above~$\lambda$.
Assume it is true for some depth $n$. Suppose we have reached a node $\rho \in T$ of length~$n$ such that $B$
is $g$-small above $\rho$. By $h$-bushiness of $T$, $\rho$ has at least $h(n)$ immediate
extensions. If~$B$ were $g$-big above $g(n)$-many immediate extensions of $\rho$
then $B$ would be also $g$-big above~$\rho$, by $g$-closedness. Hence $B$ is $g$-small above at least $h(n) - g(n)$
immediate extensions of $\rho$. It follows easily that, conditional to having reached $\rho$, the probability to avoid~$B$ at the next level is at least $1-g(n)/h(n)$. This finishes the induction. 
\end{proof}

%\begin{remark}
%Lemma~\ref{lem:small-vs-measure} establishes a useful relation between
%smallness and measure. In particular, if $h \gg g$, then the probability
%that a path avoids~$B$ is close to 1. Note that while ``small implies large probability to be avoided", the converse is not true. For example, if every node of the tree has exactly half of its successors in~$B$, then the probability that an infinite path avoids~$B$ is~$0$, but such a $B$ must be $(h/2)$-big, and a fortiori $g$-big for any $g \in  o(h)$ (up to changing finitely many values of~$g$). 
%\end{remark}

\begin{remark}
Lemma~\ref{lem:small-vs-measure} makes no computability assumption on~$T$ and~$B$. However, when~$T$ is computable, taking a path of~$T$ at random can be performed using a probabilistic algorithm, which will then produce a path avoiding~$B$ with probability at least $\prod_{i=0}^{n-1} \big(1-g(i)/h(i)\big)$ (and this still makes no computability assumption about~$B$). 
\end{remark}

Now we see how randomness helps us compute DNC functions: take a computable function~$h$ such that $\prod_i (1-2/h(i)) > 0$ (which is equivalent to $\sum_i 1/h(i) < \infty$; for example, $h(n)=(n+1)^2$ will do) and take an $h$-bushy tree~$T$. Now, take a path of~$T$ at random. Since, as we saw above, $B_{DNC}$ is 2-small in~$T$, the previous lemma tells us that the probability to get a path of the tree which avoids~$B$, and thus is a DNC function, is at least $\prod_i (1-2/h(i))$, which is positive.

%due to the following equality:
%\[
%\prod_{i \in \omega}\p{1 - e_i} = \exp\p{- \sum_{i \in \omega} \ln(1-e_i)}
%\]

%Given a computable function $g$ and a string $\sigma$, the decision of the statement
%``$B$ is $g$-big/small above $\sigma$'' is given for free when the forcing does not care
%about effectiveness. However in our case, we need to be able to decide such sentences effectively.
%It is easy to see that bigness of $B$ is a $B$-c.e. question. This is problematic
%because we may see too late that there exists a finite tree witnessing the bigness
%of our set $B$ and may have already taken the wrong decision. This is why we need to combine
%bushy tree forcing with a probabilistic tool known as the \emph{fireworks argument}.

%For a computable tree~$T$, a node $\sigma$ of~$T$, a computable order~$g$, and a computable or c.e. set of strings~$B$, the statement ``$B$ is $g$-big in~$T$ above $\sigma$" is a $\Sigma_1$ statement, uniformly in a code for $T, g, \si

\subsection{Fireworks arguments}\label{sec:fireworks}
% !TEX root =  ../doc.tex

A fireworks argument can be seen as ``probabilistic forcing"  for $\Sigma_1$ properties. 
It is best illustrated by the following theorem, due to Kautz: there exists a probabilistic 
algorithm which with positive probability produces a $1$-generic real (more precisely, 
every 2-random real computes a 1-generic real). 
Let us present the argument in a more abstract way so as to better fit the setting of 
the next section, where we will (implicitly) force with bushy trees. Let $(\mathbb{P}, \leq)$ 
be a computable partial order, and let $W_1$, $W_2$, ... be a list of uniformly c.e.\ subsets of~$\mathbb{P}$. We want to get a probabilistic algorithm to generate an infinite list 
$p_0 \geq p_1 \geq p_2 \geq ...$ such that for every set $W_i$, 
the requirement $\mathcal{R}_i$ holds, where $\mathcal{R}_i$ says that there exists 
a~$j$ such that either $p_j \in W_i$ holds or for every $q \leq p_j$, $q \notin W_i$. (For example, to get Kautz's result, one takes for $\mathbb{P}$ the set of finite strings, where $\sigma \leq \tau$ if $\sigma$ extends $\tau$ and the $W_i$ are all c.e.\ sets of strings. In this paper, $\mathbb{P}$ will typically consist of a set of finitely represented bushy trees, and $T' \leq T$ will mean that $T'$ is a subset of~$T$). 

Suppose we have already built the sequence of $p_j$ up to some $p_k$, and we want to satisfy the requirement~$\mathcal{R}_i$. If we did not care about the effectivity of the construction, we could easily satisfy the requirement by distinguishing two cases:

\begin{itemize}
\item[] Case 1 (the $\Pi_1$ case): there is no $p' \leq p_k$ such that $p' \in W_i$. In this case nothing needs to be done, $\mathcal{R}_i$ is already satisfied!
\item[]Case 2 (the $\Sigma_1$ case): there is some $p' \leq p_k$ such that $p' \in W_i$. Here it suffices to search for such a $p'$ (which can even be done effectively since $W_i$ is c.e.), and set $p_{k+1}=p'$ after which $\mathcal{R}_i$ is satisfied. 
\end{itemize}

Although both cases only require effective actions (do nothing or effectively search for a $p'$, respectively), the problem is that one cannot computably distinguish between them, as being in Case~2 is only a c.e.\  event (hence the name `$\Sigma_1$ case'). And indeed in general there is no deterministic algorithm to build a sequence of $p_j$ satisfying all requirements (otherwise one could, for example, computably build a $1$-generic). 

There \emph{is} however, a \emph{probabilistic algorithm} which builds such a sequence of $p_j$'s with positive probability. It works as follows. We start with any $p_0 \in \mathbb{P}$. Next, for each~$i$, we pick an integer~$n_i$ at random between~$1$ and $N(i)$, where~$N$ is a fixed computable function (the faster~$N$ grows, the higher the probability of success of the algorithm will be).  Moreover, for each~$i$, we let $c_i$ be a counter, initialized at~$0$, which will count how many wrong passive guesses we have made for requirement $\mathcal{R}_i$.

% and $k_i$ another counter which remembers where the last guess was made, also initialized at~$0$. 
% Note that by definition of the counters~$c_i$, we initially make this assumption for all~$\phi_i$

By `passive guess', we mean that in the construction of the $p_j$'s, we assume at some step~$k$ that we are in the $\Pi_1$ case, i.e., that no $q \leq p_k$ is in $W_i$. It is a passive guess because as we saw, if it is indeed true, requirement $\mathcal{R}_i$ is already satisfied and no particular action is needed. Of course, this guess may be incorrect but since $W_i$ is c.e., if it is incorrect we will discover this at some later stage $k'$ of the algorithm. When this happens, we make a new assumption that there are no $q \leq p_{k'}$ in~$W_i$ and so on. If at any point we make a correct passive guess, requirement $\mathcal{R}_i$ is satisfied. There is however a danger that \emph{all} the passive guesses we make for requirement $\mathcal{R}_i$ turn out to be wrong. What we do is use the number $n_i$ as a cap on how many times we allow ourselves to make a wrong passive guess for $\mathcal{R}_i$. If for some~$i$ the cap $n_i$ is reached at stage~$k$, we then make the opposite guess (``active guess"), i.e., that there \emph{is} a $q \leq p_k$ such that $q \in W_i$ holds, try to find such a $q$ and take it as our $p_{k+1}$, thus satisfying requirement~$\req{i}$. This guess is ``active" because we need to find such a~$q$ before doing anything else. But at least, as we said above, since $W_i$ is a c.e.\ set, such a~$q$ can be effectively found if it exists. Then we take $p_{k+1}=q$.

This active/passive guessing strategy is still not guaranteed to work, as one bad case remains: if we make an \emph{incorrect active guess} for some~$\mathcal{R}_i$, we then get stuck while waiting for a~$q$ in~$W_i$ which we will never find. However, this is the only bad case: if it does not happen for any~$\mathcal{R}_i$, then the algorithm succeeds in producing a sequence $p_0 \geq p_1 \geq p_2 \geq ...$ as wanted. Indeed, for every~$i$, either it makes a good passive guess for $\mathcal{R}_i$ that never turns out to be wrong, meaning that for some $p_s$, no $q \leq p_s$ is in $W_i$, or it makes a good active guess that some $q \leq p_s$ is in~$W_i$, eventually finds such a~$q$, and take it as an extension.

Why can the probability of success of the algorithm be made arbitrarily close to~1? The reason is the following key observation: For all~$i$, if all~$\{n_j: j \not= i\}$ are fixed, then there is at most one value of the random variable~$n_i$ for which we get stuck in a loop while trying to satisfy requirement~$\mathcal{R}_i$. Indeed, suppose we get stuck having chosen a value $n_i$. This means that we made $n_i-1$ incorrect passive guesses and then one incorrect active guess. Any other choice $n'_i<n_i$ would have been fine, because our $n'_i$-th guess would then have been a correct active guess and a~$q$ in~$W_i$. And any other choice $n'_i>n_i$ would also have been fine, as in this case our $n_i$-th guess would have been a correct passive guess. Thus, the probability to get stuck because of requirement~$\mathcal{R}_i$ is at most $1/N(i)$, giving a total probability of success of the algorithm of at least $1-\sum_i 1/N(i)$, which can be made arbitrarily close to~$1$ for a well chosen~$N$. \\

Now the last thing we need to check is how much randomness (in the sense of algorithmic randomness) is needed for this probabilistic algorithm to work. Let us explain what we mean. A probabilistic algorithm is nothing but a Turing functional~$\Gamma$ with access to a `random' (in the classical sense) oracle~$R \in \cs$. This is the~$R$ used by the algorithm to make its random decisions. What we have argued above is that the failure set of the algorithm
\[
U = \{ R \mid \Gamma^R \text{ is either undefined or fails to satisfy some~$\mathcal{R}_i$ } \}
\]
has probability $<1$ (by `undefined' we mean that the algorithm does not produce an infinite sequence of $p_j$'s).\\

In fact, this probability can be made arbitrarily small, therefore fireworks arguments do not give us one algorithm but a uniform family of algorithms: For any given integer~$m$, one can design, uniformly in~$m$, a probabilistic algorithm which fails with probability at most $2^{-m}$ (it suffices to choose the function~$N$ such that $\sum 1/N(i) < 2^{-m}$). Call $\Gamma_m$ the corresponding algorithm, and consider
\[
U_m = \{ R \mid \Gamma_m^R \text{ is either undefined or fails to satisfy some~$\mathcal{R}_i$ } \}
\]
which is of measure at most $2^{-m}$. The set $\bigcap_m U_m$, which is the set of $R$'s on which \emph{all} the algorithms $\Gamma_m$ fail is a null set. This means that if~$R$ is `sufficiently random', in the sense of effective randomness, it does not belong to all~$U_m$ and thus some algorithm $\Gamma_m$ succeeds using~$R$. Which level of algorithmic randomness is actually needed? One should observe that every~$U_m$ is in fact an effectively open  set relatively to~$\emptyset'$. Indeed, as we have seen, the only case the algorithm $\Gamma_m$ fails is when it waits in vain for a $q$ extending some condition~$p$ and belonging to some~$W_i$. If such a situation happens, it does so at some finite stage, i.e., having used only a finite initial segment of~$R$, hence $U_m$ is open. Moreover, testing whether the algorithm $\Gamma_m$ is stuck at a given stage can be done using~$\emptyset'$: indeed, the predicate $[(\forall q \leq p) \; q \notin W_i]$ is a $\Pi_1$-predicate uniformly in $p$ and~$i$. 

Thus, $(U_m)_{m \in \N}$ is a $\emptyset'$-Martin-L\"of test, which shows that every $2$-random real computes, via some functional $\Gamma_m$, an infinite sequence $p_0 \geq p_1 \geq p_2 \geq ...$ such that for every $i$ there is a~$j$ such that either $p_j \in W_i$ holds or for every $q \leq p_j$, $q \notin W_i$, as wanted.

\section{Main result}
% !TEX root = ../doc.tex

We shall now see how to combine fireworks arguments with bushy tree forcing to prove Theorem~\ref{thm:main-theorem}. We first provide an informal presentation of the proof. Full details will be given in the next section.

\subsection{Proof overview}\label{sec:overview}

%There exists a direct mapping between functions over integers and
%the set of infinite sequences $\omega^\omega$ by interpreting the $i$th integer of the sequence
%as the $i$th value of the associated function.
%This is why it is convenient to work with infinite subtrees of $\omega^{<\omega}$
%where each path will be a candidate function.

For this construction, we will need a hierarchy of very fast growing computable functions
\[
g_0 \ll g_1 \ll g_2 \ll \ldots
\]
($g \ll g'$ is an informal notation: it means that $g'$ grows `much faster' than~$g$) %with the property that $g_n(i)=0$ if $i<n$, thus ensuring that $\sum_i g_i$ is a computable function taking finite values. 
and another fast-growing function~$h$ (which is meant to grow faster that all the $g_i$ but with certain restrictions). At this point, we do not specify precisely what functions $g_i$ and $h$ we take. We will see during the construction which properties they need to have to make the argument work. \\

Contrary to most bushy tree arguments, the whole construction will happen within a single tree~$T$, which is exactly~$h$-bushy:
\[
T = \{\sigma \in \omega^{<\omega} \; : \; (\forall i)\; \sigma(i) < h(i)\}
\]
Typically, a bushy tree forcing argument constructs a sequence $T_0 \supseteq T_1 \supseteq \ldots$ of bushy trees, and the path obtained by forcing is in the intersection of all of those. We will not need such a sequence in our argument. However, some steps of the construction can be understood as ``locally" taking a subtree of~$T$. What we keep from other bushy tree arguments is the idea of maintaining during the construction a small set of bad strings to be avoided. But again, there is a difference in our construction: to build a DNC$^X$ function by forcing, one usually starts with the initial small bad set
\[
B_{\DNC} = \set{\sigma \in \omega^{<\omega} : (\exists e < \card{\sigma})\, \phi^X_e(e) \downarrow = \sigma(e) }
\]
We will not do this as we need our bad set of strings to be c.e.\ at all times. Our fireworks argument will (with high probability) build an~$A$ which does not compute any Martin-L\"of random real. It is only an \emph{a posteriori} analysis of the construction which will allow us to conclude that~$A$ is also DNC$^X$ with high probability. In the absence of any other requirement, we would just build $A$ value by value, picking for each~$n$ the value of $A(n)$ at random between~$0$ and $h(n)-1$. This would give us a probability of avoiding~$B_{\DNC}$ of at least $\prod_n (1-1/h(n))$, which is positive for $h$ fast-growing.  But of course there are other requirements our construction needs to meet, namely all the requirements of the form:

\begin{center}
$\req{\Gamma,d}$: either $\Gamma^A$ is partial or there is an~$n$ such that $K(\Gamma^A \uh n) < n-d$
\end{center}
where~$d$ is an integer, $\Gamma$ is a Turing functional from $\omega^{<\omega}$ to~$\cs$, and $K$ is the prefix-free Kolmogorov complexity function. (Here we use the Levin-Schnorr theorem that a real $Y$ is Martin-L\"of random if and only if for some constant~$d$, and all~$n$, $K(Y \uh n) \geq n-d$). \\

Let us see how we would ideally like to satisfy such a (single) requirement. Suppose we have already built some string~$\sigma \in T$ of length~$k$ and consider the set
\[
S = \set{\tau \in T : |\Gamma^\tau|  \geq h(k-1)}
\]
where $|\Gamma^\tau|  \geq h(k-1)$ means that the Turing reduction~$\Gamma$ produces at least~$h(k)$ bits of output on input~$\tau$, and distinguish two cases:

\begin{itemize}
\item Case 1: $S$ is $g_k$-small above~$\sigma$. In this case there is essentially nothing to worry about. We can just continue to build $A$ by making random choices. The probability that we hit the set $S$ at some point is small, namely it is at most $1-\prod_{i \geq k+1} (1-g_k(i)/h(i))$ (by Lemma~\ref{lem:small-vs-measure}), and recall that $h(i) \gg g_{k+1}(i) \gg g_k(i)$ for $i \geq k+1$. And if we do not hit~$S$, then $\Gamma^A$ will end up being partial, therefore satisfying $\req{\Gamma,d}$. 

\item Case 2: $S$ is $g_k$-big above~$\sigma$. Each element~$\tau \in S$ is such that $|\Gamma^\tau|  \geq h(k-1)$, therefore we can decompose~$S$ as
\[
S = \bigcup_{|\rho|=h(k-1)} S_{\rho}~  ~ \text{where} ~ ~ S_\rho = \{\tau \in S \mid \Gamma^\tau \ext \rho\}
\]
There are $2^{h(k-1)}$ strings of length~$h(k-1)$, therefore by Lemma~\ref{lem:smallness-add}, there must be a $\rho^*$ such that $S_{\rho^*}$ is $\left(g_{k}/2^{h(k-1)}\right)$-big above~$\sigma$ (note that in this expression, $g_k$ is a function while $h(k-1)$ is just an integer). Since being big is a $\Sigma_1$-property, such a~$\rho^*$ can be found effectively knowing~$\sigma$ and $\Gamma$, and thus the first such $\rho^*$ found with this effective search must satisfy
\begin{equation}\label{eq:basic-kolmo-bound}
K(\rho^*) \leq K(\sigma) + K(\Gamma) + c \leq 2 \sum_{i \leq k-1} \log h(i) + K(\Gamma) + c'
\end{equation}
for some fixed constants~$c, c'$ (the last term is due to the fact that $\sigma$ is a list of $k-1$ integers such that the $i$-th integer is less than $h(i)$, therefore has complexity less than $2\log h(i)$). Since $|\rho^*|=h(k-1)$ we have 
\[
K(\rho^*) - |\rho^*|  \leq 2 \sum_{i \leq k-1} \log h(i) - h(k-1) + K(\Gamma) + c'
\]

If~$h$ grows fast enough, and~$k$ is sufficiently large then this last inequality implies $K(\rho^*) \leq |\rho^*|-d$. Thus, any~$A$ which passes through a node in~$S_{\rho^*}$ satisfies requirement~$\req{\Gamma,d}$. Moreover, since~$S_{\rho^*}$ is $\left(g_{k}/2^{h(k-1)}\right)$-big above $\sigma$, this means by definition that there is a finite $\left(g_{k}/2^{h(k-1)}\right)$-bushy tree $T' \subseteq T$ of stem~$\sigma$ all of whose leaves are in~$S_{\rho^*}$. Then, what we can do is effectively find the tree $T'$ and temporarily restrict our random walk to $T'$ (picking at each step the next value at random among nodes of $T'$) until we reach a leaf of~$T'$. This guarantees the satisfaction of $\req{\Gamma,d}$ and the probability that we hit $B_{\DNC}$ during this temporary restriction is less than $1-\prod_{i \geq k} (1-2^{h(k-1)}/g_k(i))$, which can be made small if~$g_k$ is well-chosen.   \\
\end{itemize}

Of course, the problem is again that, having built $\sigma$, we cannot effectively determine whether we are in Case 1 or in Case 2. This is where the fireworks argument comes into play. We are going to pick a number $n=n(\Gamma,d)$ between $1$ and some large $N=N(\Gamma,d)$, and assume up to $n$ times that we are in Case 1 (the $\Pi_1$ case), and if proven wrong $n$ times, we will then wait until proven that we are in Case 2 (the $\Sigma_1$ case), and if so, implement the above strategy for Case 2. As with other fireworks arguments, the probability that we decide to wait at the wrong moment is at most~$1/N$, which we can thus make arbitrarily small. \\

These considerations are enough to give us a strategy which ensures, with arbitrarily high probability, the satisfaction of a \emph{single} requirement $\req{\Gamma,d}$ while avoiding the set $B_{\DNC}$ with high probability. However, there is a subtle point to address when we try to satisfy several requirements in parallel. Indeed, what can happen is that the strategy for a first requirement has made the assumption that some set $S$ is $g_k$-small in~$T$ - and thus small probability of being hit - while a strategy for a second requirement needs to make a temporary restriction to a tree $T'$. While the probability to hit~$S$ was small for a random walk within~$T$, it could happen that the random walk restricted to $T'$ has a much greater probability to hit~$S$. This is what the assumption $g_i \ll g_j$  for $i<j$ takes care of: Whenever a strategy needs to make such a restriction to a $\left( g_k/2^{h(k-1)}\right)$-bushy subtree, we will have that $g_k/2^{h(k-1)}$ grows much faster than the $g_j$'s previously considered in the proof, and thus, if a set $S$ is $g_j$-small, it will still be unlikely that we hit~$S$ while choosing a path of a $\left( g_k/2^{h(k-1)}\right)$-bushy tree at random for any $k>j$.

\subsection{The full algorithm and formal proof of correctness}

We now state the precise theorem regarding the probabilistic algorithm discussed above. The analysis of the level of effective randomness required will be done separately. 

\begin{theorem}
Let $X \in \cs$ and $h$ be a sufficiently fast-growing computable function. For every rational $\varepsilon>0$, one can effectively design a probabilistic algorithm (= Turing machine with random oracle) which, with probability at least $1-\varepsilon$, produces an $A \in \omega^\omega$ such that (1) $A$ is $\DNC^X_h$ and (2) $A$ computes no Martin-L\"of random real. 
\end{theorem}

Let~$m$ be the smallest integer such that $2^{-m+3} < \varepsilon$. Let us first define the functions $g_k$ and $h$ we alluded to above. Set $g_0$ to be the function defined by $g_0(n)=2^{n+m}$. Then, for all~$k>1$, inductively define 
\[
g_k(i) = \left\{  \begin{array}{l} 1~ \text{ if } ~  i<k \\  g_{k-1}(i) \cdot 2^{g_{k-1}(k-1)+i+m} ~ \text{otherwise} \end{array}  \right.
\]

Finally define for all~$k$
\[
h(k) = g_k(k) 
\]

Let us now give the details of the algorithm. First, number all the requirements $\req{\Gamma,d}$ and call $\req{i}$ the $i$-th requirement. As usual, we organize them in a way that all requirements receive attention infinitely often, and only one requirement receives attention at any given stage. We will see during the verification that a small extra assumption should be added, namely that every requirement should be considered for the first time at some `late enough' stage. \\

\noindent \textbf{Stage 0: Initialization}. The first thing we do is pick for all~$i$ a number $n_i$ at random between $1$ and $2^{i+m}$. Then, we initialize $\sigma$ to be the empty string. For all~$i$, set a counter~$c_i$ originally equal to~$0$.\\

% and for all~$i$, each strategy for requirement $\req{i}=\req{\Gamma,d}$ makes the assumption that 
%\[
%S = \set{\tau \in T : \tau \ext \sigma ~\text{and}~ |\Gamma^\tau|  \geq h(k)+K(\Gamma)+c}
%\]
%($c$ to be defined during the verification)  is $g_k$-small above $\sigma$\\

\noindent \textbf{Loop (to be repeated indefinitely)}. Suppose that the values $\sigma(0), ..., \sigma(k-1)$ of $\sigma$ are already defined (the last one being between $0$ and $h(k-1)-1$). Assume some requirement $\req{i}=\req{\Gamma,d}$ receives attention. 

\begin{itemize}
\item[(a)] If this requirement receives attention for the first time, we make for this requirement the assumption that the set
\[
S = \set{\tau \in T : |\Gamma^\tau|  \geq h(k-1)}
\]
is $g_k$-small above the current~$\sigma$ (again, note that this is a $\Sigma_1$ assumption so if it is false it will be discovered to be so at some finite stage).

\item[(b)] If it does not receive attention for the first time, we check whether the current assumption made for this requirement still appears to be true at stage~$k$.

	\begin{itemize}
		\item[(1)] If it does, we maintain this assumption and simply pick the value of~$\sigma(k)$ at random between $0$ and $h(k)-1$.
		\item[(2)] If the assumption is discovered to be false, we increase our `error counter' $c_i$ by~$1$.
			\begin{itemize}
				\item[(i)] If the new value of $c_i$ remains less than $n_i$, we forget our previous assumption for requirement $\req{i}$ and make a new assumption: we now assume that the set $S = \set{\tau \in T : |\Gamma^\tau|  \geq h(k-1)}$ is $g_k$-small above (the current)~$\sigma$.
 
				\item[(ii)] If the new value of $c_i$ is equal to $n_i$, we then wait until we find, for some~$\rho$ of length~$h(k-1)$, a set $S_\rho = \{\tau \in S \mid \Gamma^\tau \ext \rho\}$  which is $\left(g_k/2^{h(k-1)}\right)$-big above~$\sigma$. When this happens, i.e., when we find a finite subtree~$T'$ of $T$ of stem~$\sigma$ which is $\left(g_k/2^{h(k-1)}\right)$-bushy above~$\sigma$ and all of whose leaves are in~$S_\rho$, we choose the next values of~$\sigma$ by a downward random walk restricted to~$T'$ until we reach a leaf of~$T'$ (note that we may never find such a tree in which case our algorithm gets stuck at this stage and thus fails to even return an infinite sequence). When a leaf is reached, we mark the $i$-th requirement as satisfied. 
			\end{itemize} 
	\end{itemize}

\end{itemize}

The sequence~$A$ returned by the algorithm is the minimal element of $\omega^{\leq \omega}$ extending all the values taken by $\sigma$ throughout the algorithm. We now turn to the verification of our algorithm, which we have already done for the most part in Section~\ref{sec:overview}. We do this via a series of claims.\\

\begin{claim}
The probability that the algorithm gets stuck at some substage of type (b.2.ii) (waiting to find a big tree $T'$ which does not exist) is at most $2^{-m+1}$. 
\end{claim}

\begin{proof}
This is standard fireworks calculation: all other randomly chosen values being fixed, there is at most one value of $n_i$ which causes the algorithm to get stuck at (b.2.ii) because of requirement $\req{i}$. And since $n_i$ is chosen randomly between~$1$ and $2^{i+m}$, the probability that the algorithm gets stuck at (a.2.ii) because of requirement $\req{i}$ is at most $2^{-i-m}$. Thus, over all requirements, this gives a probability of at most $2^{-m+1}$. 
\end{proof}

\begin{claim}
Conditionally to our algorithm returning an infinite sequence, the probability that all requirements are met is at least $1-2^{-m+2}$.
\end{claim}

\begin{proof}
Let us look at a given requirement~$\req{i}=\req{\Gamma,d}$. This requirement receives attention infinitely often until satisfied. This means that if the algorithm does not get stuck, one of the following happens
\begin{itemize}
\item at some point it makes for $\req{i}$ a correct assumption during substep (a) or (b.2.i) or,
\item the $i$-th requirement causes the algorithm to enter some substep (b.2.ii) but a tree~$T'$ is found thereafter. 
\end{itemize}

These two cases are mutually exclusive. In the first case, for some~$k$ a set 
\[
S=\set{\tau \in T : |\Gamma^\tau|  \geq h(k-1)}
\]

 is correctly assumed to be $g_k$-small above the current~$\sigma$. For any later stage of the algorithm (i.e., at stages $i \geq k$), the value of $\sigma(i)$ is chosen  at random among $\tilde{h}(i)$ values, where $\tilde{h}(i)$ is either equal to $h(i)$ or to $g_{k'}(i)/2^{h(k'-1)}$ for some $k'>k$ in case some other strategy has caused a temporary restriction of the tree. The latter quantity is the smaller of the two, and by definition of the $g_k$'s, it follows that $\tilde{h}(i) \geq 2^{i+m} g_k(i)$. 

By the calculations of the proof of Lemma~\ref{lem:small-vs-measure}, it follows that in this case, the probability of hitting a node of~$S$ at some later stage is at most
\[
1 - \prod_{i \geq k} \left( 1 - \frac{g_k(i)}{\tilde{h}(i)} \right) \leq 1 - \prod_{i \geq k} \left( 1 - \frac{g_k(i)}{2^{i+m} g_k(i)} \right) \leq \sum_{i \geq k} 2^{-i-m} \leq 2^{-k-m+1}
\]

In the second case the requirement $\req{i}$ is always satisfied. Indeed, in this case, we find a string~$\rho^*$ of length~$h(k-1)$ and finite tree~$T'$ whose leaves are contained in $S_{\rho^*} = \{\tau \in S \mid \Gamma^\tau \ext \rho^*\}$  and then make a random walk within~$T'$ until we reach a leaf. As explained in the previous section, we then have
\[
K(\rho^*) \leq K(\sigma) + K(\Gamma) + c \leq 2 \sum_{i \leq k-1} \log h(i) + c' + K(\Gamma)
\]
for some fixed constants $c,c'$. But $|\rho^*| = h(k-1)$, so

\begin{equation}\label{eq:kolmo-bound}
K(\rho^*) - |\rho^*| \leq 2 \sum_{i \leq k-1} \log h(i) + K(\Gamma) - h(k-1) + c''
\end{equation}
for some fixed $c''$. By construction of~$h$, $h(k-1) - 2 \sum_{i \leq k-1} \log h(i)$ tends to $\infty$, thus for~$k$ large enough (and this `large enough' can be found computably), the value of above expression is less than $-d$, which means that the requirement $\req{\Gamma,d}$ is satisfied as soon as we reach a leaf of~$T'$. We thus add the technical extra assumption that requirement $\req{\Gamma,d}$ is only allowed to receive attention at stage~$k$ if in the above expression the right-hand side is smaller than $-d$. This essentially changes nothing since it only prevents every requirement to receive attention for finitely many stages. \\

Given a requirement~$\req{\Gamma,d}$, we say that \emph{stage~$k$ is good for~$\req{\Gamma,d}$} when after having built~$\sigma \uh k$, in the next iteration of the loop, $\req{\Gamma,d}$ receives attention and either a true assumption is made at steps (a) or (b.2), or stage (b.2.ii) is reached and a tree~$T'$ is found. To every requirement corresponds exactly one good stage (and no two requirements have a good stage in common). As we have just argued, if $k$ is the good stage for a requirement, the probability that the requirement is satisfied, conditional to the algorithm returning an infinite sequence, is at least $1-2^{-k-m+1}$. Over all requirements, this gives a probability of at least $1-\sum_k 2^{-k-m+1}=1-2^{-m+2}$.   
\end{proof}

\begin{claim}
The probability that we hit a node of $B_{\DNC}$ during the algorithm is at most $2^{-m+1}$. 
\end{claim}

\begin{proof}
There is at most one `bad' value of $\sigma(n)$ the algorithm can choose (namely, $\phi^X_n(n)$, if it is defined). Whenever a value $\sigma(n)$  is chosen at random for some~$n$, it is either chosen at random between~$0$ and $h(n)-1$ or, in case of a temporary restriction to a subtree, between~$0$ and $g_k(n)/2^{h(k-1)}-1$ for some~$k \leq n$ (in case of a temporary restriction of the tree). Both quantities are at least  $2^{n+m}$, by construction. This gives a total probability of at most $\sum_{n} 2^{-n-m} = 2^{-m+1}$ of hitting~$B_{\DNC}$. 
\end{proof}

The theorem immediately follows from the three claims: the probability that an infinite sequence~$A$ is returned and all requirement are satisfied is at least $1-2^{-m+1}-2^{-m+2}-2^{-m+1}=1-2^{-m+3} \geq 1-\varepsilon$.

\subsection{How much randomness do we need?}

It remains to conduct, like in Section~\ref{sec:fireworks}, an analysis of the level of algorithmic randomness needed to make the algorithm work. The attentive reader will notice that there are two uses of randomness in the construction: the first one to choose the sequence $(n_i)$ which will make the fireworks argument work (i.e., the algorithm won't get stuck), and the second one which helps choosing a node of the tree at random during the construction. For a given~$m$, consider the algorithm~$\Gamma_m$ with probability of success at least $1-2^{-m}$. There are three ways in which it can fail:
\begin{enumerate}
\item It could get stuck at some stage b.2.ii. 
\item It could hit a node which belongs to $B_{\DNC}$.
\item It could make at some stage~$k$ a true assumption that a set~$S$ is $g_k$-small, but nonetheless hit a node of~$S$ later on (when it hits a node of the set~$S$ corresponding to a wrong assumption, this is not a problem because the assumption will be discovered to be wrong later on and a new assumption will be made for the requirement). 
\end{enumerate}
 
Technically, the occurrence of the third case does not necessarily mean that the algorithm has failed, but if neither of these three cases occur the algorithm succeeds, as explained above. The total probability of such events is at most $2^{-m}$. Moreover, if any event of the above three types happens, it does so at some finite stage, thus after having used only finitely many bits of the random oracle. The open set of oracles that cause events of type 1 and 3 to happen can be effectively enumerated relatively to $\emptyset'$. Indeed, for the first type this is exactly what is explained in Section~\ref{sec:fireworks}, namely that $\emptyset'$ can check at any given given stage whether the algorithm is stuck at stage b.2.ii. The third type can also be checked using $\emptyset'$: indeed, the sets we assume to be $g_k$-small are c.e.\ sets and since~$g_k$ is computable uniformly in~$k$, the smallness can be checked using~$\emptyset'$ and checking whether a node chosen at some stage of the algorithm is in~$S$ can also be done effectively in $\emptyset'$ (since $S$ is c.e.). Thus we can design a $\emptyset'$-Martin-L\"of test $(\mathcal{U}^{\emptyset'}_m)_{m \in \N}$ such that $\mathcal{U}^{\emptyset'}_m$ covers the set of oracles which make the algorithm $\Gamma_m$ fail because of cases 1 or 3.

The second type of failure is even easier to analyse. The set $B_{\DNC}$ is $X$-c.e., so the set of oracles which cause the algorithm to hit a node of $B_{\DNC}$ is effectively open relative to~$X$. Thus we can design a $X$-Martin-L\"of test $(\mathcal{V}^{X}_m)_{m \in \N}$ such that $\mathcal{V}^{X}_m$ covers the set of oracles which make the algorithm $\Gamma_m$ fail because of case 2.

This finishes the proof of Theorem~\ref{thm:main-theorem}: if $Z$ is $X$-random and 2-random, for~$m$ large enough it will be outside $\mathcal{U}^{\emptyset'}_m$ and outside $\mathcal{V}^{X}_m$, hence the algorithm $\Gamma_m$ will succeed on input~$Z$.

\section{Further results}
% !TEX root =  ../doc.tex

The proof of Theorem~\ref{thm:main-theorem} can be adapted to prove more results on the class of DNC functions in the Levin-V'yugin algebra. For example, in this proof, we construct a real of $\DNC^X$ degree which computes no Martin-L\"of random real, but we do so using Kolmogorov complexity in a rather liberal way. By being very slightly more precise we can get a stronger result.

\begin{theorem} \label{thm:dnc-h0-vs-dnc-h1}
Let $X \in \cs$. For every fixed computable function~$h_0$, for every sufficiently fast-growing computable $h$, every real~$Z$ which is both $X$-Martin-L\"of random and 2-random computes a $\DNC_{h}^X$ function which computes no $\DNC_{h_0}$ function. 
\end{theorem}

This is a stronger theorem than Theorem~\ref{thm:main-theorem} because, as we saw in Section 1.1, when $h_0$ is sufficiently fast-growing, every Martin-L\"of random real computes a $\DNC_{h_0}$ function. Again, the fact that $\dnc_{h_0}$ is strictly contained in $\dnc_{h_1}$ when $h_1$ grows sufficiently faster than~$h_0$ is known (see~\cite{mushfeq}), but our theorem shows that this separation holds in the Levin-V'yugin algebra as well. \\

The first thing to do to adapt our previous proof is to use the relationship between Kolmogorov complexity and DNC functions discovered by Kjos-Hanssen et al.~\cite{Kjos-HanssenMS2011}. For a function~$h: \N \rightarrow \N$, call \emph{$h$-complex} a real~$A \in \cs$ such that $K(X \uh h(n)) \geq n-O(1)$. Call a real \emph{complex} if it is $h$-complex for some computable~$h$.  Kjos-Hanssen et al.\ proved that a real computes a $\DNC$ function with computable bound if and only if it computes a complex real. More precisely: for any computable~$h_0$, for any computable~$h_1$ which grows sufficiently faster than $h_0$, if a real~$A$ computes a $\DNC_{h_0}$ function, it computes an $h_1$-complex real, and if $A$ computes an $h_0$-complex real, it computes a $\DNC_{h_1}$ function.

\begin{proof}[Proof of Theorem~\ref{thm:dnc-h0-vs-dnc-h1}]
By the correspondence between DNC functions and complex reals, it suffices to show the following: For every fixed computable function~$h_0$, for every sufficiently fast-growing $h$, every real~$Z$ which is both $X$-Martin-L\"of random and 2-random computes a $\DNC_{h}^X$ function which computes no $h_0$-complex function. We modify the proof of Theorem~\ref{thm:main-theorem} as follows. The requirements now become:
\begin{center}
$\req{\Gamma,d}$: either $\Gamma^A$ is partial or there is an~$n$ such that $K(\Gamma^A \uh h_0(n)) < n-d$
\end{center}

The new functions $g_i$'s are defined by 
\[
g_k(i) = \left\{  \begin{array}{l} 1~ \text{ if } ~  i<k \\  g_{k-1}(i) \cdot 2^{h_0(g_{k-1}(k-1))+i+m} ~ \text{otherwise} \end{array}  \right.
\]
and again, $h(k)=g_k(k)$ for all~$k$. The sets~$S$ considered in Step (a) of the algorithm are now 
\[
S = \set{\tau \in T : |\Gamma^\tau|  \geq h_0(h(k-1))}
\]

The rest of the construction remains the same. The estimate~\eqref{eq:basic-kolmo-bound} is left unchanged by this modification, but we now have $|\rho^*| \geq h_0(h(k-1))$. Together with~\eqref{eq:basic-kolmo-bound}, for~$k$ large enough, this guarantees the satisfaction of $\req{\Gamma,d}$ (where $\Gamma$ is the reduction with respect to which~$\rho^*$ is defined). 
\end{proof}

Using another adaptation of the proof of Theorem~\ref{thm:main-theorem}, we can also transfer to the Levin-V'yugin algebra the following result, due to Miller (see \cite[section 3]{mushfeq}): there exists a $\DNC$ function which computes no complex real. Namely, the following holds.

\begin{theorem} \label{thm:dnc-vs-complex}
Let $X \in \cs$. Every real~$Z$ which is both $X$-Martin-L\"of random and 2-random computes a $\DNC^X$ function which computes no $\DNC_{h}$ function for any computable~$h$. 
\end{theorem}

Although the general structure is similar, this second adaptation is not straightforward and quite a number of important changes are needed. The first thing to notice is that there is obviously no hope to conduct the whole construction in an $h$-bushy tree for a computable~$h$ since we want~$A$ to compute no complex real, which is equivalent to computing no computably bounded $\DNC$ function. Thus we will need to work in the full tree $\omega^\omega$ and at each level~$k$, choose dynamically the interval $[0,h(k)]$ for the random choice of~$\sigma(k)$.

For this construction, the requirements are of the form:

\begin{center}
$\req{\Gamma,\phi,d}$: $\Gamma^A$ is partial, or $\phi$ is partial, or there is an~$n$ such that $K(\Gamma^A \uh \phi(n)) \leq n- d$
\end{center}

where the $\Gamma$'s are still Turing functionals from $\omega^\omega$ to $\cs$, and the $\phi$'s are partial computable functions from $\N$ to $\N$. 
How can we satisfy a single requirement $\req{\Gamma,\phi,d}$? Again, suppose some string~$\sigma$ of length~$k$ has already been built, and consider the set
\[
S = \set{\tau \in 2^{<\omega} : |\Gamma^\tau|  \geq \phi(r+d)}
\]
where~$r$ is an upper bound for the Kolmogorov complexity of $(\Gamma,d,\phi,\sigma)$ (which can be found computably since there are computable upper bounds of prefix-free Kolmogorov complexity). By convention, this set is empty if~$\phi(r+d)$ is undefined. Again, let us analyze the different cases and how to succeed in each of them. 

\begin{itemize}
\item Case 1: $\phi(r+d)$ is undefined. Then the requirement is satisfied vacuously.

\item Case 2: $\phi(r+d)$ is defined and $S$ is $\left(2^{\phi(r+d)} \cdot g_0\right)$-small above~$\sigma$, where $g_0$ is the function defined in previous constructions. In this case, it suffices to choose a function $g_1 \gg 2^{\phi(r+d)} \cdot g_0$ and for all $k' \geq k$ pick the value of $\sigma(k')$ at random between $0$ and $g_1(k')$. By smallness of $S$, using Lemma~\ref{lem:small-vs-measure} as usual, we will avoid the set~$S$ with high probability, thus satisfying requirement~$\req{\Gamma,\phi,d}$.

\item Case 3: $\phi(r+d)$ is defined and $S$ is $\left(2^{\phi(r+d)} \cdot g_0\right)$-big above~$\sigma$. Each element~$\tau \in S$ is such that $|\Gamma^\tau|  \geq \phi(r+d)$, therefore we can decompose~$S$ as
\[
S = \bigcup_{|\rho|= \phi(r+d)} S_{\rho}~  ~ \text{where} ~ ~ S_\rho = \{\tau \in S \mid \Gamma^\tau \ext \rho\}
\]
and since there are $2^{\phi(r+d)}$ strings of length~$\phi(r+d)$, there must be a $\rho^*$ such that $S_{\rho^*}$ is $g_0$-big above~$\sigma$ and such a~$\rho^*$ can be found effectively knowing~$(\Gamma,d,\phi,\sigma)$, hence by the choice of~$r$, 
\begin{equation*}
K(\rho^*) \leq r 
\end{equation*}
We have $|\rho^*|=\phi(r+d)$, so for~$m$ large enough, this guarantees $K(\rho^* \uh \phi(r+d)) \leq r$, thus satisfying requirement $\req{\Gamma,\phi,d}$ with $n=r+d$.

\end{itemize}

We want to use a fireworks argument to help us choose between these three cases, but some care is needed since we no longer have a $\Sigma_1/\Pi_1$ dichotomy. Indeed Case 2 is neither $\Sigma_1$ nor~$\Pi_1$. The solution is to introduce a priority ordering over passive guesses. We will first make a number of assumptions that $\phi(r+d)$ is undefined at different stages. If all of these assumptions turn out to be wrong (= the error counter reaches its cap), we will then make the assumption that $\phi(r+d)$ is defined at the current stage, wait for the value $\phi(r+d)$ to be defined, and only then make \emph{one} assumption that the current~$S$ is $\left(2^{\phi(r+d)} \cdot g_0\right)$-small above the current~$\sigma$. If proven wrong, we will begin another round of assumptions that $\phi(r+d)$ is undefined (using a new error counter with a new cap) before making a new `Case 2' assumption. Finally, when many of these `Case 2' assumptions are proven to be wrong, we make one last assumption, a `Case 3' assumption, and if everything goes well we will satisfy the requirement $\req{\Gamma,\phi,d}$. Thus, for any given requirement, our sequence of assumptions will look like this:
\begin{center}
C1, C1, ..., C1, C2, C1, ..., C1, C2, C1, ..., C1, C2, ..........., C2, C1, C1, ..., C1, C3 
\end{center}
(where C$i$=Case $i$) unless one of the C1/C2 assumptions is never proven to be wrong, in which case we succeed. This time there are two possible ways for the algorithm to get stuck: either wrongly assume in Case 2 that $\phi(r+d)$ is defined, and then wait forever for it to converge, or like in the previous proofs, get stuck because of a wrong C3 assumption, waiting in vain to find a big subtree with leaves in a given c.e.\ set. The probability of either of these events happening can be made arbitrarily small by a fireworks argument. \\

The idea to handle several requirements at the same time is similar to what was done in our previous constructions, but this time is dynamic: Before making a C2/C3 assumption of type `the following set~$S$ is small/big', we need to dynamically decide what `big/small' should mean. What we do is first look at what other smallness assumptions are currently being made for other requirements. If there are $l$ current assumptions of type `$S_i$ is $g_i$-small' for $i \leq l$, then we first choose a function $G$ much larger than $g_1+....+g_{l}$ and evaluate the smallness of~$S$ in terms of this new function. In case $S$ is then assumed to be small it is just added to the list of current assumptions. In case it is correctly assumed to be big, since $G$ is much larger than the other $g_i$'s, the probability that we hit one of the $S_i$ during the temporary restriction of the tree will be, as in the previous proofs, close to~$0$. \\

While we hope that the reader is already convinced at this point, we provide the formal details for completeness. 

\begin{theorem}
Let $X \in \cs$ and $h$ be a sufficiently fast-growing computable function. For every rational $\varepsilon>0$, one can effectively design a probabilistic algorithm which, with probability at least $1-\varepsilon$, produces an $A \in \omega^\omega$ such that (1) $A$ is $\DNC^X$ and (2) $A$ computes no complex real. 
\end{theorem}

Again, let us take $m$ to be the smallest integer such that $2^{-m+3} < \varepsilon$. Our algorithm is the following. 

\noindent \textbf{Stage 0: Initialization}. 
First, for each requirement $\req{i}$, pick a number $n(i, 1)$ at random between $1$ and $2^{\langle i, 1, 0 \rangle +m}$.
The number $n(i,1)$ is meant to be a cap for the number of wrong C2 assumptions for requirement~$\req{i}$.
Moreover, for each integer $b$ between $0$ and $n(i,1)$, pick a number $n(i, 2, b)$ at random between  $1$
and $2^{\langle i, 2, b \rangle +m}$. Each time C2 makes a wrong assumption, C1 starts a new series of assumptions
with $n(i, 2, b)$ as a new cap, where $b$ is the number of wrong C2 assumptions.
Create two counters $c_i,c'_i$, initialized at~$0$ ($c_i$ counts the number of wrong C2 assumptions for requirement $\req{i}$ and $c'_i$ counts the number of wrong C1 assumption during the current run of such assumptions for requirement~$\req{i}$). Let $\mathcal{L}$ be a list of assumptions (coded as integers, with at most one assumption per requirement), initially empty. Finally, initialize $\sigma$ to be the empty string. 

%First, for each triple of natural integers $(i,a,b)$, with $a \in \{1,2\}$ pick a number $n(i,a,b)$ at random between $1$ and $2^{\langle i,a,b \rangle+m}$. The number $n(i,1,b)$ is meant to be a cap for the number of wrong C1 assumptions for requirement $\req{i}$ during the $b$-th run of such assumptions, and $n(i,2,0)$ will be the cap for the number of wrong C2 assumptions for requirement~$\req{i}$.  For every~$i$, create two counters $c_i,c'_i$, initialized at~$0$ ($c_i$ counts the number of wrong C2 assumptions for requirement $\req{i}$ and $c'_i$ counts the number of wrong C1 assumption during the current run of such assumptions for requirement~$\req{i}$). Let $\mathcal{L}$ be a list of assumptions (coded as integers, with at most one assumption per requirement), initially empty. Finally, initialize $\sigma$ to be the empty string. 
% and for all~$i$, each strategy for requirement $\req{i}=\req{\Gamma,d}$ makes the assumption that 
%\[
%S = \set{\tau \in T : \tau \ext \sigma ~\text{and}~ |\Gamma^\tau|  \geq h(k)+K(\Gamma)+c}
%\]
%($c$ to be defined during the verification)  is $g_k$-small above $\sigma$\\

\noindent \textbf{Loop (to be repeated indefinitely)}. Suppose that the values $\sigma(0), ..., \sigma(k-1)$ of $\sigma$ are already defined and some requirement $\req{i}=\req{\Gamma,\phi,d}$ receives attention.  Let~$r$ be an upper bound of the Kolmogorov complexity of the current tuple~$(\sigma,\Gamma,\phi,d,\mathcal{L})$. Let $g_0, ..., g_l$ be the computable functions such that an assumption `$S$ is $g_i$-small' is currently in~$\mathcal{L}$. We (locally) define a function $G$ by $G(u)=2^{u+m} (g_0(u)+\ldots+g_l(u))$.

\begin{itemize}
\item[(a)] If this requirement receives attention for the first time, we make for this requirement the assumption that $\phi(r+d)$ is undefined, and add this assumption to $\mathcal{L}$.
%\[
%S = \set{\tau \in T : \tau \ext \sigma ~\text{and}~ |\Gamma^\tau|  \geq h(k)+m}
%\]
%is $g_k$-small above the current~$\sigma$ (again, note that this is a $\Sigma_1$ assumption so if it is false it will be discovered to be so at some finite stage).

\item[(b)] If it does not receive attention for the first time, we check whether the current assumption made for this requirement still appears to be true at stage~$k$.

	\begin{itemize}
		\item[(1)] If it does, we maintain this assumption and simply pick the value of~$\sigma(k)$ at random between $0$ and $G(k)$.
		\item[(2)] If the assumption is discovered to be false, and it was a C1 assumption ($\phi$~undefined on some value), we remove this assumption from~$\mathcal{L}$ and increase $c'_i$ by~$1$. 
		
			\begin{itemize}
				\item[(i)] If the new value of $c'_i$ remains less than $n(i,2,c_i)$, we make a new assumption: we now assume that $\phi(r+d)$ is undefined and add this assumption to $\mathcal{L}$. 
				\item[(ii)] If the new value of $c'_i$ reaches $n(i,2,c_i)$, we wait for $\phi(r+d)$ to become defined. We then make the assumption that the set $S = \{\tau \in 2^{<\omega} \mid |\Gamma^\tau|\geq \phi(r+d)\}$ is $(2^{\phi(r+d)} \cdot G)$-small above~$\sigma$ and add this assumption to~$\mathcal{L}$. 
			\end{itemize}

		\item[(3)] If the assumption is discovered to be false, and it was a C2 assumption (smallness of some set~$S$), we remove this assumption from~$\mathcal{L}$, and increase $c_i$ by~$1$.

			\begin{itemize}
				\item[(i)] If the new value of $c_i$ remains less than $n(i,1)$, we make a new assumption: we now assume that $\phi(r+d)$ is undefined, add this assumption to $\mathcal{L}$, and reset~$c'_i$ to~$0$.  
				\item[(ii)] If the new value of $c_i$ is equal to $n(i,1)$, we then wait until we find, for some~$\rho$ of length~$\phi(r+d)$, a set $S_\rho = \{\tau \in S \mid \Gamma^\tau \ext \rho\}$  which is $G$-big above~$\sigma$. When this happens, i.e., when we find a finite tree~$T$ of stem~$\sigma$ which is $G$-bushy above~$\sigma$, we choose the next values of~$\sigma$ by a downward random walk restricted to~$T$ until we reach a leaf of~$T$. When a leaf is reached, we mark the $i$-th requirement as satisfied. 
			\end{itemize} 
	\end{itemize}

\end{itemize}

The sequence~$A$ returned by the algorithm is the minimal element of $\omega^{\leq \omega}$ extending all the values taken by $\sigma$ throughout the algorithm. 

The verification is similar to the previous proofs.

\begin{claim}
The probability that the algorithm gets stuck at some substage of type (b.2.ii) or (b.3.ii) is at most $2^{-m+1}$. 
\end{claim}

\begin{proof}
Indeed, by the standard verification of fireworks arguments, the probability that the algorithm gets stuck at the end of the $b$-th run of C1 assumptions for requirement~$i$ (making a bad assumption that $\phi$ is defined on some value) is, by construction, $2^{-\langle i,2,b \rangle -m}$. Likewise, the probability that it gets stuck because of a bad C3 assumption for requirement~$i$ is $2^{-\langle i,1, 0 \rangle -m}$. Since the sum of the terms $2^{-\langle i,a,b \rangle -m}$ is $2^{-m+1}$, we have the desired result. 
\end{proof}

\begin{claim}
Conditionally to algorithm returning an infinite sequence, the probability that all requirements are met during the construction is at least $1-2^{-m+2}$.
\end{claim}

\begin{proof}
Fix a requirement~$\req{i}=\req{\Gamma,\phi,d}$. This requirement receives attention infinitely often until satisfied. This means that if the algorithm does not get stuck, one of the following happens
\begin{itemize}
\item At some point it makes for $\req{i}$ a correct assumption that $\phi$ is undefined on some value during substep (a) or (b.3.i). 
\item At some point it makes for $\req{i}$ a correct assumption that $S = \{\tau \in 2^{<\omega} \mid |\Gamma^\tau|\geq \phi(r+d)\}$ is $(2^{\phi(r+d)} \cdot g)$-small above~$\sigma$ for some computable~$g$. 
\item the $i$-th requirement causes the algorithm to enter some substep (b.3.ii) but a tree~$T$ is found thereafter. 
\end{itemize}

In the first case, the requirement is satisfied vacuously. In the second case, we have a set $S = \{\tau \in 2^{<\omega} \mid |\Gamma^\tau|\geq \phi(r+d)\}$ which is correctly assumed to be $(2^{\phi(r+d)} \cdot g)$-small above~$\sigma$ for some computable~$g$. For any $i \geq k$, the value of $\sigma(i)$ is chosen at later stages of the algorithm at random among at least $G(i)$-many values, where~$G$ is chosen to be greater than $2^{i+m} \cdot h$ for any function~$h$ appearing in the list~$\mathcal{L}$, which in particular includes~$g$ (since the assumption featuring~$g$ is correct, it is never removed from the list). By Lemma~\ref{lem:small-vs-measure}, it follows that the probability of hitting a node of~$S$ at some later stage is at most
\[
1 - \prod_{i \geq k} \left( 1 - \frac{g(i)}{2^{i+m}g(i)} \right) = 1 - \prod_{i \geq k} \left( 1 - \frac{1}{2^{i+m}} \right) \leq \sum_{i \geq k} 2^{-i-m} \leq 2^{-k-m+1}
\]

In third case, we find a string~$\rho^*$ of length~$\phi(r+d)$ and $G$-bushy finite tree~$T$ whose leaves are contained in $S_{\rho^*} = \{\tau \in S \mid \Gamma^\tau \ext \rho^*\}$. The string~$\rho^*$ can be effectively found knowing $\Gamma, \phi, d$, the current value of $\sigma$ and $G$. But $G$ can be computed knowing the list~$\mathcal{L}$, and the integer~$r$ is precisely defined to bound the complexity of the tuple~$(\sigma,\Gamma,\phi,d,\mathcal{L})$, thus
\[
K(\rho^*) \leq r
\]
But $|\rho^*| = \phi(r+d)$, so the requirement $\req{\Gamma,\phi,d}$ is satisfied for $n=r+d$. Note that to be completely rigorous, when we say that `$r$ bounds the Kolmogorov complexity of~$(\sigma,\Gamma,\phi,d,\mathcal{L})$', we also need~$r$ to be large enough to overcome the additive constants that will arise in the calculation of the bound of $K(\rho^*)$. This is done, as usual, by picking~$r$ `large enough' (which can be achieved computably), or using a fixed-point argument.

Now, given a requirement~$\req{\Gamma,\phi,d}$, we say that \emph{stage~$k$ is good for~$\req{\Gamma,d}$} when after having built~$\sigma \uh k$, the next iteration of the loop, $\req{\Gamma,\phi,d}$ receives attention and a true assumption is made or stage (b.3.ii) is reached and a tree~$T$ is found. To every requirement corresponds exactly one good stage (and no two requirements have a good stage in common). As before, if $k$ is the good stage for a requirement, the probability that the requirement is satisfied, conditional to the algorithm returning an infinite sequence, is at least $1-2^{-k-m+1}$ and over all requirements, this gives a probability of at least $1-\sum_k 2^{-k-m+1}=1-2^{-m+2}$. 
\end{proof}

\begin{claim}
The probability that we hit a node of $B_{\DNC}$  is at most $2^{-m+1}$. 
\end{claim}

\begin{proof}
For every~$n$ the value of~$\sigma(n)$ is chosen at random among at least $2^{n+m}$ values, so the probability of picking $\phi_n^X(n)$, if defined, is at most $2^{-n-m}$, thus a total bound of $\sum_n 2^{-n-m} = 2^{-m+1}$. 
\end{proof}

Thus the probability of success of the algorithm is at least $1-2^{-m+3}$. To get Theorem~\ref{thm:dnc-vs-complex}, it remains to evaluate the level of algorithmic randomness needed, but the situation is essentially the same as in previous proofs. Checking whether the algorithm gets stuck at a given stage can be tested effectively in~$\emptyset'$ (with $\emptyset'$ we can check whether partial functions are defined or not, and we can check bigness/smalllness of c.e.\ sets of strings) and whether it hits $B_{\DNC}$ can be tested using~$X$ as discussed before. \\

%\section[Exists k DNR(k) does not imply WWKL]
%  {$(\exists k)(\forall f)\dnr(k,f)$ does not imply $\wwkl$}
%\input{parts/existsk_dnrk}

\section{Conclusion}
% !TEX root =  ../doc.tex

The above results provide a clear picture of the hierarchy of DNC notions in V'Yugin's algebra. Namely, for every computable function $h_0$ which grows fast enough, and computable $h_1$ which grows fast enough compared to~$h_0$, we have the strict inclusions:

\[
\mlr \sqsubset \dnc_{h_0} \sqsubset \dnc_{h_1} \sqsubset \complex \sqsubset \dnc
\]

Moreover, every 2-random real computes a witness for each of the four separations. While this is an interesting result in and of itself, the techniques we employed to prove these results are without doubt the main contribution of this paper. They illustrate the power and flexibility of fireworks arguments and show that they interact well with complex forcing notions. What is also interesting is that `true randomness' seems to be needed for our fireworks arguments. Other known fireworks arguments, such as  the proof that every 2-random real computes a hyperimmune set or that every 2-random computes a 1-generic real require only a very weak form of randomness. Indeed, for both of these constructions, Barmpalias et al.~\cite{BarmpaliasDL2013} showed that it suffices to have a non-computable real which is \emph{Turing-below} a 2-random, despite the fact that a real which is simply below a 2-random may have very little randomness content. Having a non-computable real which is below a 2-random real is not sufficient in our case: Indeed no 1-generic real computes a DNC function (see for example~\cite{DowneyH2010}), and as we just said, there are 1-generic reals which are below a 2-random real. \\

We believe that fireworks arguments will yield more applications in the future. In the constructions of this paper, the strategies for different requirements do interact, but there is no injury per se. It would be very interesting to find examples of situations where fireworks arguments can be mixed with finite/infinite injury constructions. 

\vspace*{1cm}

\noindent \textbf{Acknowledgements}. We are thankful to Peter G\'acs and Alexander Shen for enlightening us on fireworks arguments, to Mushfeq Khan for having made early versions of his survey available to us, and to Chris Porter and Rupert H\"olzl for having introduced us to the Levin-V'Yugin algebra and asked the question which ultimately led to this paper. We would also to thank two anonymous referees for their many useful comments and suggestions. Both authors acknowledge the support of the John Templeton Foundation through the grant
``Structure and Randomness in the Theory of Computation."

\bibliographystyle{plain}
\bibliography{doc}

\begin{thebibliography}{10}

\bibitem{AmbosSpiesKLS2004}
Klaus Ambos-Spies, Bj{\o}rn Kjos-Hanssen, Steffen Lempp, and Theodore~A.
  Slaman.
\newblock Comparing {DNR} and {WWKL}.
\newblock {\em Journal of Symbolic Logic}, 69(4):1089--1104, 2004.

\bibitem{BarmpaliasDL2013}
George Barmpalias, Adam Day, and Andy Lewis-Pye.
\newblock The typical {T}uring degree.
\newblock {\em Proceedings of the London Mathematical Society}, 109(1):1--39,
  2013.

\bibitem{BienvenuP2016}
Laurent Bienvenu and Christopher~P. Porter.
\newblock Deep {$\Pi^0_1$} classes.
\newblock {\em Bulletin of Symbolic Logic}, 22(2):249--286, 2016.

\bibitem{cai2011elements}
Mingzhong Cai.
\newblock {\em Elements of Classical Recursion Theory: Degree-Theoretic
  Properties and Combinatorial Properties}.
\newblock PhD thesis, Cornell University, 2011.

\bibitem{DeLeeuwMSS1956}
Karel de~Leeuw, Edward~F. Moore, Claude Shannon, and Norman Shapiro.
\newblock Computability by probabilistic machines.
\newblock In {\em Automata Studies}. Princeton University Press, 1956.

\bibitem{Demuth1988}
Oswald Demuth.
\newblock Remarks on the structure of tt-degrees based on constructive measure
  theory.
\newblock {\em Commentationes Mathematicae Universitatis Carolinae},
  29(2):233--247, 1988.

\bibitem{dorais2014comparing}
Fran{\c{c}}ois~G. Dorais, Jeffry~L. Hirst, and Paul Shafer.
\newblock Comparing the strength of diagonally non-recursive functions in the
  absence of {$\Sigma^0_2$} induction.
\newblock {\em Journal of Symbolic Logic}, 80:1211--1235, 2015.

\bibitem{DowneyH2010}
Rodney Downey and Denis Hirschfeldt.
\newblock {\em Algorithmic randomness and complexity}.
\newblock Theory and Applications of Computability. Springer, 2010.

\bibitem{greenberg2011diagonally}
Noam Greenberg and Joseph~S. Miller.
\newblock Diagonally non-recursive functions and effective {H}ausdorff
  dimension.
\newblock {\em Bulletin of the London Mathematical Society}, 43(4):636--654,
  2011.

\bibitem{JockuschS1972}
Carl Jockusch and Robert Soare.
\newblock {$\Pi^0_1$} classes and degrees of theories.
\newblock {\em Transactions of the American Mathematical Society}, 173:33--56,
  1972.

\bibitem{Kautz1991}
Steven~M. Kautz.
\newblock {\em Degrees of random sequences}.
\newblock PhD thesis, Cornell University, 1991.

\bibitem{Khan2013}
Mushfeq Khan.
\newblock Shift-complex sequences.
\newblock {\em Bulletin of Symbolic Logic}, 19(2):199--215, 2013.

\bibitem{mushfeq}
Mushfeq Khan and Joseph~S. Miller.
\newblock Forcing with bushy trees.
\newblock arXiv preprint (arXiv:1503.08870), 2015.

\bibitem{Kjos-HanssenMS2011}
Bj{\o}rn Kjos-Hanssen, Wolfgang Merkle, and Frank Stephan.
\newblock Kolmogorov complexity and the recursion theorem.
\newblock {\em Transactions of the American Mathematical Society},
  363(10):5465--5480, 2011.

\bibitem{Kucera1985}
Antonin Ku{\v c}era.
\newblock Measure, {$\Pi^0_1$} classes, and complete extensions of {PA}.
\newblock {\em Lecture Notes in Mathematics}, 1141:245--259, 1985.

\bibitem{KumabeL2009}
Masahiro Kumabe and Andrew~{E. M.} Lewis.
\newblock A fixed-point-free minimal degree.
\newblock {\em Journal of the London Mathematical Society}, 80(3):785--797,
  2009.

\bibitem{LevinV1977}
Leonid Levin and Vladimir V'yugin.
\newblock Invariant properties of informational bulks.
\newblock {\em Lecture Notes in Computer Science}, 53:359--364, 1977.

\bibitem{Paris1977}
Jeffrey Paris.
\newblock Measure and minimal degrees.
\newblock {\em Annals of Mathematical Logic}, 11:203--216, 1977.

\bibitem{RumyantsevS2013}
Andrei Rumyantsev and Alexander Shen.
\newblock Probabilistic constructions of computable objects and a computable
  version of {L}ov\'asz local lemma.
\newblock arXiv preprint (arXiv:1305.1535), 2013.

\bibitem{Rumyantsev2011}
Andrey~Yu. Rumyantsev.
\newblock Everywhere complex sequences and the probabilistic method.
\newblock In {\em STACS}, volume~9 of {\em LIPIcs}, pages 464--471, 2011.

\bibitem{Simpson2011}
Stephen~G. Simpson.
\newblock Mass problems associated with effectively closed sets.
\newblock {\em Tohoku Mathematical Journal}, 63:489--517, 2011.

\bibitem{Vyugin1976}
Vladimir~V. V'yugin.
\newblock On {T}uring invariant sets.
\newblock {\em Soviet Mathematics Doklady}, 17:1090--1094, 1976.

\bibitem{Vyugin1982}
Vladimir~V. V'yugin.
\newblock Algebra of invariant properties of binary sequences.
\newblock {\em Problemy Peredachi Informatsii}, 18(2):83--100, 1982.

\end{thebibliography}

\end{document}